\documentclass[12pt]{article}

\usepackage{amssymb}
\usepackage{amsmath} 
\usepackage{amsfonts}
 
\usepackage{amsthm}

\DeclareMathAlphabet{\mathpzc}{OT1}{pzc}{m}{it}

\theoremstyle{plain} 
\newtheorem{theorem}{Theorem}[section] 
\newtheorem{lemma}[theorem]{Lemma}

\newtheorem{question}[theorem]{Question}

\theoremstyle{definition}

\theoremstyle{remark} 
\newtheorem{remark}{Remark}

\makeatletter
{\newcount\@hour}
{\newcount\@minute}
\def\timenow{\@hour=\time \divide\@hour by 60
\number\@hour:
  \multiply\@hour by 60 \@minute=\time
  \global\advance\@minute by -\@hour
  \ifnum\@minute<10 0\number\@minute\else
  \number\@minute\fi}
\def\ctimenow{\hfil{\tt \jobname.tex, \today~Time: \timenow }\hfil}
      \let\@oddfoot\ctimenow\let\@evenfoot\ctimenow
\makeatother

\begin{document}

\begin{center}{\bf \Large
Borel complexity of the family of attractors for weak IFSs
}
\end{center}
\smallskip
\begin{center}
By
\end{center}
\smallskip
\begin{center} Pawe\l{} Klinga and Adam Kwela
\end{center}

\begin{abstract}
This  paper  is  an  attempt  to  measure  the  difference between  the  family  of  iterated function systems  attractors  and  a  broader  family,  the  set of attractors for weak iterated function systems. We discuss Borel complexity of the set wIFS$^d$ of attractors for weak iterated function systems acting on $[0,1]^d$ (as a subset of the hyperspace $K([0,1]^d)$ of all compact subsets of $[0,1]^d$ equipped in the Hausdorff metric). We prove that wIFS$^d$ is $G_{\delta\sigma}$-hard in $K([0,1]^d)$, for all $d\in\mathbb{N}$. In particular, wIFS$^d$ is not $F_{\sigma\delta}$ (in contrast to the family IFS$^d$ of attractors for classical iterated function systems acting on $[0,1]^d$, which is $F_{\sigma}$). Moreover, we show that in the one-dimensional case, wIFS$^1$ is an analytic subset of $K([0,1])$.
\let\thefootnote\relax
\let\thefootnote\relax
\end{abstract}

\section{Introduction}

In \cite[Theorem 3.9]{DS1} (see also \cite{DS2}) E. D'Aniello and T.H. Steele proved that the family IFS$^d$ of attractors for iterated function systems acting on a metric space $[0,1]^d,$ $d\in\mathbb{N},$ is a meager $F_\sigma$ subset of the space $K([0,1]^d)$ of compact subsets of $[0,1]^d$ equipped with the Hausdorff metric (note that in \cite[Theorem 5]{SS} it is stated that IFS$^d$ is a nowhere dense $F_\sigma$ subset of $K([0,1]^d)$, which is not true as IFS$^d$ is dense -- it contains all finite subsets of $[0,1]^d$; cf. \cite[Lemma 4.3]{DS3}). In \cite[Theorem 4.1]{my} it is shown that the family wIFS$^d$ of attractors for weak iterated function systems acting on $[0,1]^d,$ $d\in\mathbb{N},$ is also meager. Thus, wIFS$^d$ contains the meager $F_\sigma$ set IFS$^d$ and is contained in another meager $F_\sigma$ subset of $K([0,1]^d)$.

Our main aim  is  to  measure  the  difference between  the  families IFS$^d$ and wIFS$^d$. This paper is an attempt to calculate the descriptive complexity of wIFS$^d$. We prove that wIFS$^1$ is a $G_{\delta\sigma}$-hard analytic subset of $K([0,1])$ and wIFS$^d$, for $d>1$, is a $G_{\delta\sigma}$-hard subset of $K([0,1]^d)$. In particular, wIFS$^d$ is not $F_{\sigma\delta}$, which shows that it is a more complicated object than IFS$^d$. 

In Section 2 we provide some background needed in our considerations. All results are proved in Section 3.

\section{Preliminaries}

\subsection{Binary sequences}

By $\{0,1\}^\mathbb{N}$ we denote the family of all sequences $\alpha:\mathbb{N}\to\{0,1\}$. Similarly, $\{0,1\}^n$ is the family of all sequences $\alpha:\{1,\ldots,n\}\to\{0,1\}$ of length $n$ and we use the symbol $\{0,1\}^{<\mathbb{N}}$ for the family of all finite sequences of zeros and ones, i.e., $\{0,1\}^{<\mathbb{N}}=\bigcup_n \{0,1\}^n$.

If $s$ is a finite sequence of real numbers then by $lh(s)$ we denote its length. If $s=(s_1,\ldots,s_n)\in\mathbb{R}^n$ (so $lh(s)=n$) and $k\leq lh(s)$ then $s|k=(s_1,\ldots,s_k)$ is the restriction of $s$ to $k$. 

\subsection{IFS attractors}

By $K(X)$ we denote the collection of all compact subsets of a compact metric space $(X,d)$. Throughout the paper we will assume that $X=[0,1]^d$ for some $d\in\mathbb{N}$. However, the rest of this subsection is true for any compact $X$.

A function $f\colon X \to X$ is called a contraction on $X$ whenever there exists a constant $L\in [0,1)$ such that for every $x,y\in X$ it is true that $d(f(x),f(y))\leqslant L\cdot d(x,y)$. 

Every finite set $\{ s_1, \dots, s_k \}$ of contractions acting on $X$ will be called an iterated function system or, in short, IFS. By $\mathcal{S}\colon K(X) \to K(X)$ we denote the Hutchinson operator for $\{ s_1, \dots, s_k \}$, i.e.
$$\mathcal{S}(A) = \bigcup_{i=1}^k s_i[A].$$
If $(X,d)$ is complete, then so is $(K(X),d_H)$, where $d_H$ is the Hausdorff metric. Thus, by applying the Banach fixed-point theorem, the equation $\mathcal{S}(A) = A$ has a unique compact solution. Such set is called the attractor for the iterated function system $\{ s_1, \dots, s_k \}$. The form of the Hutchinson operator justifies why attractors are often used to describe fractals, since many well known self-similar sets satisfy $\mathcal{S}(A)=A$ when the contractions comprising $S$ are linear.

By IFS$^d$ we will denote the set of all attractors for iterated function systems acting on $[0,1]^d$, i.e., a compact set $A\subseteq[0,1]^d$ is in IFS$^d$ if there is an iterated function system acting on $[0,1]^d$ such that $\mathcal{S}(A) = A$.

\subsection{Weak IFS attractors}

A finite set of functions $\{s_1,\dots,s_k\}$ acting on $X$ is called a weak iterated function system, or wIFS, if for each $i\in\{1,\dots,k\}$ and every pair of distinct points $x,y\in X$ it is true that $d(s_i(x),s_i(y))<d(x,y)$. The functions themselves will be called weak contractions. A compact set $A\subseteq X$ satisfying $A=\bigcup_{i=1}^k s_i[A]$ is called an attractor for $\{s_1,\dots,s_k\}$. By \cite{E}, if $X$ is compact then every weak iterated function system has a unique attractor.

Obviously, each attractor of some IFS is an attractor of some wIFS. However, the reverse inclusion is not true -- see \cite{NFM} for counterexamples.

By wIFS$^d$ we will denote the set of all attractors for weak iterated function systems acting on $[0,1]^d$, i.e., a compact set $A\subseteq[0,1]^d$ is in wIFS$^d$ if there is a weak iterated function system $\{s_1,\dots,s_k\}$ acting on $[0,1]^d$ such that $A=\bigcup_{i=1}^k s_i[A]$.

\subsection{Borel hierarchy}

Let $\Gamma$ be a Borel class (for instance $\Gamma=F_\sigma$ or $\Gamma=G_{\delta\sigma}$). We say that a subset $A$ of a Polish space $X$ is $\Gamma$-hard if for each Polish space $Y$ and each $B\in\Gamma(Y)$ there is a continuous function $f:Y\to X$ such that $f^{-1}[A]=B$. If additionally $A\in\Gamma(X)$ then we say that $A$ is $\Gamma$-complete. 

$\Gamma$-complete sets can be viewed as the most complicated members of the class $\Gamma$ -- they are in the class $\Gamma$ but not in any lower Borel class. For instance, a $G_{\delta\sigma}$-complete set cannot be $F_{\sigma\delta}$.

Observe that if $B\subseteq Y$ is $\Gamma$-complete and there is a continuous function $f:Y\to X$ such that $f^{-1}[A]=B$ then $A$ is $\Gamma$-hard.

Examples of $\Gamma$-complete sets for various Borel classes $\Gamma$ can be found in \cite{Kechris}. Such examples, especially the ones with a nice combinatorial definition, are often applied to show (with the use of the above paragraph) that some set is $\Gamma$-hard.

One particular example is especially important to us. Treating the power set $\mathcal{P}(X)$ of a set $X$ as the space $\{0,1\}^X$ of all functions $f:X\to \{0,1\}$ (equipped with the product topology, where the space $\left\{0,1\right\}$ carries the discrete topology) by identifying subsets of $X$ with their characteristic functions, allows us to talk about Borel complexity of subsets of $\mathcal{P}(X)$. The aforementioned example is:
$$\emptyset\times\text{Fin}=\left\{A\subseteq\mathbb{N}^2:\forall_n A_n\text{ is finite}\right\},$$
where $A_n=\{k\in\mathbb{N}: (n,k)\in A\}$, which is $F_{\sigma\delta}$-complete (see \cite[Corollary 3.4 and discussion below it]{Solecki}).

\section{Main results}

\begin{theorem}
wIFS$^1$ is an analytic subset of $K([0,1])$.
\end{theorem}

\begin{proof}
	We will use the following characterization: a set $A\subseteq X$ is analytic if and only if there exists a Polish space $Y$ and a continuous map  $f \colon Y \to X$ such that $f[Y]=A$ (see \cite{Kechris} or \cite{Srivastava}).
	
	Let $L_n^W$ denote the set of attractors for systems of at most $n$ weak contractions acting on $[0,1]$. Since $wIFS^1 = \bigcup_n L_n^W$, it is sufficient to show that each $L_n^W$ is analytic and use the fact that the family of analytic sets is closed under countable unions.
	
	The reasoning will be based on \cite{DS3}. Specifically, authors define the set
	$$ wLip1(X) = \{ f \in C(X,X) \colon \forall_{x,y\in X, x\neq y}\: d(f(x), f(y))<d(x,y) \}  $$
	($C(X,X)$ is the Banach space of all continuous functions from $X$ to $X$ equipped with the supremum norm). 	In other words, $wLip1(X)$ is the family of weak contractions acting on $X$. Then the authors consider the Cartesian product $(wLip1(X))^n$ and show that the following function:
	$$ f_X\colon (wLip1(X))^n \to K(X) $$
	which maps a set of $n$ weak contractions to its unique attractor, is continuous provided that $X$ is a compact metric space \cite[Lemma 4.6]{DS3}. Therefore, $f_{[0,1]}[(wLip1([0,1]))^n] = L_n^W$. To finish the proof we need to make sure that the domain is Polish -- and it is in the case of $X=[0,1]$, as in \cite[Remark 4.4 and Lemma 4.5]{DS3} the authors show that $wLip1([0,1])$ is a $G_\delta$ subset of a closed subset of $C([0,1],[0,1])$.
\end{proof}

The rest of the paper is devoted to showing that wIFS$^d$ is $G_{\delta\sigma}$-hard in $K([0,1]^d)$, for every $d\in\mathbb{N}$. We will mostly work in the one-dimensional case, as other cases will follow from it.

In the lemma below the symbols $\exists_n^\infty$ and $\forall_m^\infty$ mean ``there exist infinitely many $n$" and ``for all $m$ excluding a finite number", respectively.

\begin{lemma}\label{fsigmadelta-complete}
	The set $\{ A \subseteq \mathbb{N}^2  \colon \exists_n^\infty\: \forall_m^\infty\:\: (n,m) \in A \}$ is $F_{\sigma\delta}$-complete.
\end{lemma}

\begin{proof}
	Let us start with the following observation:
	\begin{align*}
	& \{ A \subseteq \mathbb{N}^2  \colon \exists_n^\infty\: \forall_m^\infty\:\: (n,m) \in A \} = \\
	=\:\:& \{ A \subseteq \mathbb{N}^2  \colon \forall_i\:\exists_{n\geqslant i}\:\exists_j\:\forall_{m\geqslant j} \:\: (n,m) \in A \} = \\
	=\:\:& \bigcap_i \bigcup_{n\geqslant i}\bigcup_j \bigcap_{m\geqslant j} \left\{ A \subseteq \mathbb{N}^2 \colon (n,m) \in A \right\}
	\end{align*}
and since $\left\{ A \subseteq \mathbb{N}^2 \colon (n,m) \in A \right\}$ is a member of the base of the topology, we automatically obtain that the set in question is $F_{\sigma\delta}$. It is therefore sufficient to show that it is $F_{\sigma\delta}$-hard.

For convenience let us denote
$$\{ A \colon \exists_n^\infty\:\forall_n^\infty\:\: (n,m) \notin A \} = \{A \colon \exists_n^\infty\: A_n \in Fin \} = Z,$$
where $Fin$ denotes the family of all finite subsets of $\mathbb{N}$ and $A_n$ is a vertical section of the set $A$ placed at the $n$th coordinate.

Let us define the following function $ f \colon\mathcal{P}(\mathbb{N}^2) \to \mathcal{P}(\mathbb{N}^2)$:
$$ f(A) = \left\{ (n,m)\in\mathbb{N}^2 \colon m \in \bigcup_{i\leqslant n} A_i \right\}. $$
We need to show the continuity of the function as well as 
$$ f^{-1}[Z] = \emptyset \times Fin. $$
This will finish the proof, since the map $B\mapsto B^c$ (mapping a subset of $\mathbb{N}^2$ to its complement) is clearly continuous and $B\in Z$ is equivalent to $B^c\in\{ A \subseteq \mathbb{N}^2  \colon \exists_n^\infty\: \forall_m^\infty\:\: (n,m) \in A \}$.

We will start with the inclusion $ f^{-1}[Z] \subseteq \emptyset \times Fin. $ Fix $B\in f^{-1}[Z]$. Denote $A = f[B]$. Then $A\in Z \cap f[\mathcal{P}(\mathbb{N}^2)]$. If there exists $n$ such that $A_n\notin Fin$, then for every $m\geqslant n$ we have $A_m\notin Fin$, since $A\in f[\mathcal{P}(\mathbb{N}^2)]$. Therefore $A \notin Z$, a contradiction. Hence, for every $n$ it is true that $A_n\in Fin$. If there existed $n$ such that $B_n\notin Fin$, then $\bigcup_{i\leqslant n} B_i = A_n \notin Fin$. We obtain $B \in \emptyset \times Fin.$

Let us move to the reverse inclusion. Fix $B\notin f^{-1}[Z]$. Then $f[B]\notin Z$. Let us denote $A=f[B]$. We have that there exists $n$ such that $\bigcup_{i\leqslant n} B_i = A_n \notin Fin$. Therefore there exists $i\leqslant n$ such that $B_i \notin Fin$. Hence, $B\notin \emptyset\times Fin$.

For the final part we will show the continuity of the function. To do this, we will show that the preimages of sets from the subbase are open, i.e.
\begin{enumerate}
	\item $f^{-1}[\left\{ A \colon (n,m) \in A  \right\}]$ is open.
	\item $f^{-1}[\left\{ A \colon (n,m) \notin A  \right\}]$ is open.
\end{enumerate}
For 1. notice the following:
\begin{align*}
&f^{-1}[\left\{ A \colon (n,m) \in A  \right\}] = \\
= & \left\{ B \colon \exists_{i \leqslant n}\:\:(i,m) \in B \right\} = \\
= & \bigcup_{i\leqslant n} \{B \colon (i,m) \in B \}
\end{align*}
and the last set is indeed open.
Similarly for 2.:
\begin{align*}
&f^{-1}[\left\{ A \colon (n,m) \notin A  \right\}] = \\
= & \left\{ B \colon \forall_{i \leqslant n}\:\:(i,m) \notin B \right\} = \\
= & \bigcap_{i\leqslant n} \{B \colon (i,m) \notin B \}
\end{align*}
and we obtain a finite intersection of open sets, i.e. an open set.
\end{proof}

\begin{remark}
\label{rem}
	In the proof of Lemma \ref{main_theorem} we will use a certain modification of the Cantor ternary set. It is constructed by removing from $[0,1]$ a $(0.1,0.9)$ interval and leaving two intervals on each side so that each is a tenth of the original interval. Each next step follows accordingly. We will refer to it as the ``decimal" Cantor set (in contrast to the original ternary Cantor set). Equivalently, the decimal Cantor set is given by 
	$$\left\{\sum_{n}x_n\cdot\frac{9}{10^n}:x_n\in\{0,1\}\right\}.$$
	
	We will need a certain way of addressing specific elements of this set. We will describe it now.
	
	Let $s\in \{0,1\}^{lh(s)}$ be a finite sequence. Put
	$$d_s = \sum_{1\leqslant k\leqslant lh(s)} e_{lh(s)}^{s|k} ,$$
	where $e_{i}^{s|k}$ is the $i$th term of the sequence $e^{s|k}$ and 
	$$ e^s = \begin{cases}
		(0,0,\dots)\text{, if $s$ is an empty sequence or a sequence which ends with 0,}\\
		( 0,\dots, 0, \frac{9}{10^{lh(s)}} \cdot \frac{1}{2}, \frac{9}{10^{lh(s)}} \cdot \frac{1}{4}, \dots ), \text{ otherwise},
	\end{cases}$$
	where in the other case the number of zeros in front is $lh(s)-1$.
	
	We will show the following two properties of such addressing of the decimal Cantor set:
	\begin{itemize}
		\item[(1)] for each $x\in \{0,1\}^\mathbb{N}$ we have $\sum_{n} d_{x|n} = \sum_{k, x(k) = 1} \frac{9}{10^k} $, which will show that $\{\sum_n d_{x|n}:x\in \{0,1\}^\mathbb{N}\}$ is indeed the decimal Cantor set, 
		\item[(2)] if $s,t\in\{0,1\}^{<\mathbb{N}}$, $s <_{lex} t, lh(s)=lh(t)$ then $d_s<d_t$, where $<_{lex}$ is the lexicographic order.
	\end{itemize}

Let us start by showing (1). Fix $x\in \{0,1\}^\mathbb{N}$.
$$\sum_{n} d_{x|n} = \sum_{n} \sum_{k\leqslant n} e_{n}^{x|k} = \sum_{n} \sum_{k} e_{n}^{x|k} = (*), $$
where the last equality holds because $e_n^{x|k} = 0$ for $k>n$,
$$(*) = \sum_{k} \sum_{n} e_{n}^{x|k} = \sum_{k, x(k) = 1} \frac{9}{10^k}, $$
because for those $k$ where $x(k) = 0$ the sum consists of zeros.

Let us proceed to (2). Let $m$ be the index of the first term where the sequences $s$ and $t$ differ, i.e. $s|(m-1) = t|(m-1)$ and $s_m = 0, t_m = 1$. Denote $n = lh(s) = lh(t)$.
$$  d_t-d_s =  \sum_{k<m} e_n^{t|k} + e_n^{t|m} + \sum_{m<k \leqslant n } e_n^{t|k} - \sum_{k<m} e_n^{s|k} - e_n^{s|m} - \sum_{m<k \leqslant n } e_n^{s|k}  = $$
$$ =  e_n^{t|m} + \sum_{m<k \leqslant n } e_n^{t|k} - e_n^{s|m} - \sum_{m<k \leqslant n } e_n^{s|k}  \geqslant $$
$$\geqslant \frac{9}{10^m} \cdot \frac{1}{2^{n-m+1}} - \left( \frac{9}{10^{m+1}} \cdot \frac{1}{2^{n-m}} + \frac{9}{10^{m+2}} \cdot \frac{1}{2^{n-m-1}} + \dots + \frac{9}{10^n} \cdot \frac{1}{2} \right) = (*),$$
where the inequality holds since $e_n^{t|m} = \frac{9}{10^m} \cdot \frac{1}{2^{n-m+1}} $ and the content of the parenthesis covers the extreme case where the difference would be maximal.

We need to show that $(*)>0$. It would be equivalent to:
$$ \frac{9}{10^m} \cdot \frac{1}{2^{n-m+1}} > \frac{9}{10^{m+1}} \cdot \frac{1}{2^{n-m}} + \frac{9}{10^{m+2}} \cdot \frac{1}{2^{n-m-1}} + \dots + \frac{9}{10^n} \cdot \frac{1}{2} \iff $$
$$ \frac{1}{2^{n-m+1}} > \frac{1}{10} \cdot \frac{1}{2^{n-m}} + \frac{1}{10^{2}} \cdot \frac{1}{2^{n-m-1}} + \dots + \frac{1}{10^{n-m}} \cdot \frac{1}{2} \iff $$
$$ \frac{1}{2^{n-m+1}} > \sum_{i=1}^{n-m} \left( \frac{1}{2^{n-m}}\cdot \frac{1}{10} \right)\cdot \left( \frac{2}{10} \right)^{i-1} \iff $$
$$ \frac{1}{2^{n-m+1}} >  \frac{1}{2^{n-m}}\cdot \frac{1}{10} \cdot \frac{1 - \left( \frac{ 2 }{10} \right)^{n-m} }{1-\frac{2}{10}} \iff $$
$$ \frac{1}{2^{n-m+1}} >  \frac{1}{2^{n-m}}\cdot \frac{1 - \left( \frac{ 2 }{10} \right)^{n-m} }{8} \iff $$
$$ \frac{1}{2^{n-m+1}} >  \frac{1}{2^{n-m}}\cdot \frac{1}{8} \cdot \left( \frac{10^{n-m} - 2^{n-m}}{10^{n-m}} \right) \iff $$
$$ \frac{1}{2} \cdot 8 \cdot 10^{n-m} > 10^{n-m} - 2^{n-m} \iff $$
$$ 3 \cdot 10^{n-m} > - 2^{n-m} $$
and the final inequality is true. This concludes our remark.

\end{remark}

\begin{remark}
\label{def1}
In the next lemma we will need the following two technical definitions.

Denote by $D$ the decimal Cantor set (so $D=\{\sum_l d_{x|l}:x\in \{0,1\}^\mathbb{N}\}$) and let $\alpha \in \{0,1\}^\mathbb{N} $ and $\delta\in(0,1)$. We define $g_\alpha^\delta:D\to[0,1]$ by $g_\alpha^\delta(\sum_l d_{x|l})=\sum_l d_{x|l} \delta^{(1-\alpha_l)}$ for all $x\in \{0,1\}^\mathbb{N}$. 

Moreover, if $n\in\mathbb{N}$, $X\subseteq \{0,1\}^\mathbb{N}$ and $C=\{\sum_l c_{x_l}:x\in X\}$ for some $\{c_s:s\in\{0,1\}^{<\mathbb{N}}\}\subseteq\mathbb{R}$, then we define:
\begin{itemize}
\item $E^n_1\left(C\right)=\left\{\sum_l c_{x|l}:x\in X,x_{n+1}=0,x_{n+2}=0\right\},$ 
\item $E^n_2\left(C\right)=\left\{\sum_l c_{x|l}:x\in X,x_{n+1}=0,x_{n+2}=1\right\},$ 
\item $E^n_3\left(C\right)=\left\{\sum_l c_{x|l}:x\in X,x_{n+1}=1,x_{n+2}=0\right\},$ 
\item $E^n_4\left(C\right)=\left\{\sum_l c_{x|l}:x\in X,x_{n+1}=1,x_{n+2}=1\right\}.$ 
\end{itemize}
\end{remark}

\begin{lemma}
\label{delta}
	Suppose that $C = D \cap \left[ 0, \frac{1}{10^n} \right]$ for some $n\in\mathbb{N}$ (so $C=\{\sum_l d_{x|l}:x\in \{0,1\}^\mathbb{N},x_l=0\text{ for }l\leq n\}$) and fix $\delta \in (\frac{80}{89},1)$, $k\in\mathbb{N}$ and $\alpha \in \{0,1\}^\mathbb{N} $ such that $\alpha_l = 1$ for all $l\geqslant n+k$. 
	Then for every weak contraction $f$, the set $f[C]$ can intersect at most two of the sets $E^n_i(g^\delta_\alpha[C])$ for $i\leq 4$.
Moreover, $f[C]$ can intersect only two consecutive sets of this form.
\end{lemma}

\begin{proof}	
	We will prove it inductively with respect to $k$. 
	
	In the first step fix $n,\delta,f$ and $\alpha$ such that $\alpha_l = 1$ for all $l\geqslant n+1$. 
	Then $g^\delta_\alpha[C] = C$ (as $d_{x|l}=0$ whenever $l\leq n$ and $x\in \{0,1\}^\mathbb{N}$ is such that $x_m=0$ for all $m\leq n$). Therefore either $f[C]\cap (E^n_1(g^\delta_\alpha[C])\cup E^n_2(g^\delta_\alpha[C]))=\emptyset$ or $f[C]\cap (E^n_3(g^\delta_\alpha[C])\cup E^n_4(g^\delta_\alpha[C]))=\emptyset$. 
	
	Let us proceed to $k=2$, i.e., fix $n,\delta,f$ and $\alpha$ such that $\alpha_l = 1$ for all $l\geqslant n+2$. We have two cases. Either $\alpha_l=1$ for all $l\geq n+1$, which has been covered in the previous paragraph, or $\alpha_{n+1}=0$ and $\alpha_l=1$ for all $l\geq n+2$, which we need to discuss. Denote $C_L=C\cap\left[ 0, \frac{1}{10^{n+1}} \right]$, $C_R=C\cap\left[\frac{9}{10^{n+1}},\frac{1}{10^{n}} \right]$, $g^\delta_\alpha[C_L]=E^n_1(g^\delta_\alpha[C])\cup E^n_2(g^\delta_\alpha[C])$ and $g^\delta_\alpha[C_R]=E^n_3(g^\delta_\alpha[C])\cup E^n_4(g^\delta_\alpha[C])$. Observe that neither $f[C_L]$ nor $f[C_R]$ can intersect both $g^\delta_\alpha[C_L]$ and $g^\delta_\alpha[C_R]$ as the longest gaps in $C_L$ and $C_R$ have lengths $\frac{8}{10^{n+2}}$ while the distance between $g^\delta_\alpha[C_L]$ and $g^\delta_\alpha[C_R]$ is at least $\delta\frac{8}{10^{n+1}}>\frac{8}{10^{n+2}}$ (as $\delta>\frac{80}{89}>\frac{1}{10}$). What is more, note that if $x,y\in \{0,1\}^\mathbb{N}$ are such that $x_m=y_m=0$ for all $m\leq n$ and additionally $x|n+1=y|n+1$ (i.e., $x_{n+1}=y_{n+1}$) then 
	$$\sum_l d_{x|l} \delta^{(1-\alpha_l)}-\sum_l d_{y|l} \delta^{(1-\alpha_l)}=\sum_l d_{x|l}-\sum_l d_{y|l}.$$
In other words, $C_L$, $C_R$, $g^\delta_\alpha[C_L]$ and $g^\delta_\alpha[C_R]$ all look exactly the same (they are all translations of $C_L$). In particular, $f[C_L]$ cannot intersect two sets of the form $E^n_i(g^\delta_\alpha[C])$ for $i\leq 4$ (and neither can $f[C_R]$). Thus, $f[C]$ intersects at most two of them. Observe that the gap between $C_L$ and $C_R$ is $\frac{8}{10^{n+1}}$ while the distance between $E^n_1(g^\delta_\alpha[C])$ and $E^n_3(g^\delta_\alpha[C])$ (which is the same as the distance between $E^n_2(g^\delta_\alpha[C])$ and $E^n_4(g^\delta_\alpha[C])$ and smaller than the distance between $E^n_1(g^\delta_\alpha[C])$ and $E^n_4(g^\delta_\alpha[C])$) is at least $\delta\left(\frac{8}{10^{n+1}}+\frac{9}{10^{n+2}}\right)$. Since $\delta>\frac{80}{89}$, we have 
$$\frac{8}{10^{n+1}}<\delta\left(\frac{8}{10^{n+1}}+\frac{9}{10^{n+2}}\right).$$ 
This means that $f[C]$ can only intersect two consecutive sets of the form $E^n_i(g^\delta_\alpha[C])$ for $i\leq 4$.

	Now we move to the inductive step. Suppose that the thesis is true for some $k$. Again fix $n,\delta,f$ and $\alpha$ such that $\alpha_l = 1$ for all $l\geqslant n+k+1$. Let $C_L$ and $C_R$ be as before. Assume to the contrary that $f[C_L]\cup f[C_R]$ intersects at least three sets of the form $E^n_i(g^\delta_\alpha[C])$ for $i\leq 4$. Then either $f[C_L]$ or $f[C_R]$ intersect at least two of them. Without loss of generality we may assume that $f[C_L]$ has this property. Observe that $f[C_L]$ cannot intersect both $L=E^n_1(g^\delta_\alpha[C])\cup E^n_2(g^\delta_\alpha[C])$ and $R=E^n_3(g^\delta_\alpha[C])\cup E^n_4(g^\delta_\alpha[C])$ for the same reason as before (as $\delta>\frac{80}{89}>\frac{1}{10}$). Thus, without loss of generality we may assume that $C_L\subseteq L$. It follows from the inductive assumption that $C_L$ can only intersect two consecutive sets of the form $E^{n+1}_i(L)$ for $i\leq 4$ (note that $E^n_1(g^\delta_\alpha[C])=E^{n+1}_1(L)\cup E^{n+1}_2(L)$ and $E^n_2(g^\delta_\alpha[C])=E^{n+1}_3(L)\cup E^{n+1}_4(L)$). Hence, if $f[C_L]$ intersects two sets of the form $E^n_i(g^\delta_\alpha[C])$, then it has to intersect $E^{n+1}_2(L)$ and $E^{n+1}_3(L)$. Then $f[C_R]$ cannot intersect $R$ as the distance between $E^{n+1}_3(L)$ and $R$ is at least $\delta\left(\frac{8}{10^{n+1}}+\frac{9}{10^{n+2}}\right)$ while the gap between $C_L$ and $C_R$ is $\frac{8}{10^{n+1}}$ and $\delta>\frac{80}{89}$. This finishes the proof.
\end{proof}

Now we move to the most technical part from which $G_{\delta\sigma}$-hardness of wIFS$^d$ will easily follow.

\begin{lemma}\label{main_theorem}
There is a continuous function $\varphi\colon \mathcal{P}(\mathbb{N}^2) \to K([0,1])$ such that $ \varphi^{-1}[wIFS^1] = \mathcal{P}(\mathbb{N}^2) \setminus \{ A \colon \exists_n^\infty\: \forall_m^\infty\:\: (n,m) \in A \} $. In particular, wIFS$^1$ is $G_{\delta\sigma}$-hard in $K([0,1])$.
\end{lemma}

\begin{proof}
The ``In particular" part follows from Lemma \ref{fsigmadelta-complete}.

In this proof for $A\subseteq\mathbb{R}$ and $r\in\mathbb{R}$ we will write $rA=\{r\cdot a: a\in A\}$ and $A+r=\{a+r:a\in A\}$.
	
	We will start by defining a sequence $(a_n)_n$ which we will need later. Put $a_0=1.9$. Let $C_0$ be the decimal Cantor set (see Remark \ref{rem}).
	
	Put $a_{1}'=\frac{a_0}{2}$ and $\varepsilon_{1} = \frac{1}{5} \cdot a_{1}' $. 
	
	Let us introduce a notation such that our decimal Cantor set
	$$C_0 = \bigcap_i\bigcup_{j \leqslant k_i^1 } I_i^j, $$
	where $I_i^j$ are the intervals from the construction of decimal Cantor set: $I_i^1=[0,1/10^{i-1}]$ and $ |I_i^{j_1}| = |I_i^{j_2}| $ for all $i\in\mathbb{N}$ and $j_1,j_2\leq k_i^1$ (in particular, $k_i^1=2^{i-1}$).
	
	Since $C_0$ is a measure zero set, there exists $i_1$ such that 
	$$ 1.9\cdot\sum_{ j\leqslant k_{i_1}^1 } |I_{i_1}^j| = |I_{i_1}^1| \cdot k_{i_1}^1 \cdot 1.9 < \varepsilon_1. $$
	
	To shorten the notation we will write $ \widetilde{k_{i_1}^1} = k_{i_1}^1\cdot 5$.
	
	Pick $\delta_1\in(0,1)$ such that $\delta_{1}^{\widetilde{k^{1}_{i_{1}}}-1}>\frac{190}{199}$.
	
	Denote $D_1 = C_0\cap I_{i_1}^1, C_1^1 = D_1 + a_0$ and
$$ C_1^k = \delta_1^{k-1}\cdot D_1+ \left( a_0 + | I_{i_1}^1 |\cdot 1.9 \cdot ( 1 + \delta_1 + \delta_1^2 + \dots + \delta_1^{k-2} ) \right) $$
for $k \in \{2,\dots, \widetilde{k_{i_1}^1} \} $.

Put $a_1 = | I_{i_1}^1 | \cdot 1.9 \cdot \left( 1 + \delta_1 + \dots + \delta_1^{ \widetilde{k_{i_1}^1} -1 } \right) $. Therefore we obtain
$$ a_1 < | I_{i_1}^1 | \cdot 1.9 \cdot \widetilde{k_{i_1}^1} = | I_{i_1}^1 | \cdot 1.9 \cdot k_{i_1}^1\cdot 5 < 5 \varepsilon_1 = a_1' = \frac{1}{2} a_0 .$$
Hence, $a_1 \leqslant \frac{1}{2}a_0.$

Let us assume that $a_n$ is already defined. Put $a_{n+1}'=\frac{a_n}{2}$ and $\varepsilon_{n+1} = \frac{1}{5(n+1)} \cdot a_{n+1}' $.

Consider the previously defined sets $C_l^j$ for $l\leq n$ and $j\leq \widetilde{k_{i_l}^l}$. For each such set let us divide its length by $ \delta_l^{j-1} $ and shift it so that its middle coincides with the middle of $C_l^j$. Let us denote the rescaled and shifted set $E_l^j$. Notice that each $E_1^j$ is a translation of the set $D_1=C_0\cap I_{i_1}^1$ (in particular, the diameter of each $E_1^j$ is $\frac{1}{10^{i_1-1}}$). From the next paragraphs of the construction it will follow that in fact each $E_l^j$ is a translation of $C_0\cap I_{i}^1$ for some $i\in\mathbb{N}$.

Put
$$ C_n = C_0 \cup \bigcup_{ k\leqslant \widetilde{k_{i_1}^1} } E_1^k \cup \dots \cup \bigcup_{ k\leqslant \widetilde{k_{i_n}^n} } E_n^k = C_{n-1} \cup \bigcup_{ k\leqslant \widetilde{k_{i_n}^n} } E_n^k .$$

Let $ \left( I_i^j \right) $ be such a sequence of intervals that
$$ C_n = \bigcup_i \bigcap_{j\leqslant k_i^{n+1}} I_i^j, $$
$| I_i^j |=| I_i^{j'} |$ and $I_i^1=[0,1/10^{i-1}]$ for all $i\in\mathbb{N}$ and $j_1,j_2\leq k_i^1$.

Since $C_n$ is a finite union of subsets of the decimal Cantor set, it is of measure zero. Therefore there exists $i_{n+1}$ such that
$$ 1.9\cdot\sum_{j\leqslant k_{i_{n+1}}^{n+1} } |I_{i_{n+1}}^j| = |I_{i_{n+1}}^1|\cdot k_{i_{n+1}}^{n+1} \cdot 1.9 < \varepsilon_{n+1}. $$
Put $ \widetilde{k_{i_{n+1}}^{n+1}} = k_{i_{n+1}}^{n+1} \cdot 5 \cdot (n+1)$ and $ D_{n+1} = C_0 \cap I_{i_{n+1}}^1 $. Pick $\delta_{n+1}\in(0,1)$ such that $\delta_{n+1}^{\widetilde{k^{n+1}_{i_{n+1}}}-1}>\frac{190}{199}$.

Let us put
$$ C_{n+1}^1 = D_{n+1} + (a_0+\dots+a_n) ,$$
$$ C_{n+1}^k = \delta_{n+1}^{k-1} \cdot D_{n+1} + \left( a_0+\dots+a_n + | I_{i_{n+1}}^1 | \cdot 1.9 \cdot ( 1+\delta_{n+1} + \delta_{n+1}^2 + \dots + \delta_{n+1}^{k-2} ) \right) $$
for $k\in \left\{ 2, \dots, \widetilde{k_{i_{n+1}}^{n+1}} \right\} $. Observe that sets $C_{l}^j$ defined in this way guarantee that in the next inductive step each $E_{n+1}^j$ (defined similarly as in the paragraph before the definition of $C_n$) is a translation of $C_0\cap I_{i_{n+1}}^1$. Hence, $C_{n+1}$ will have measure zero.

Finally, put
$$ a_{n+1} = | I_{i_{n+1}}^{1} | \cdot 1.9 \cdot \left( 1 + \delta_{n+1} + \cdots + \delta_{n+1}^{\widetilde{ k_{i_{n+1}}^{n+1} } - 1} \right) .$$

Similarly as before, it turns out that $a_{n+1} \leqslant \frac{a_n}{2} $, which concludes the inductive construction.

Since $a_{n+1} \leqslant \frac{a_n}{2} $ for all $n$, the series $\sum_{i=0}^\infty a_n\leq a_0\cdot \sum_{i=0}^\infty\frac{1}{2^i}<+\infty$, which implies that there exists $x$ such that
$$ \overline{C_0 \cup \bigcup_{n} \bigcup_{ k\leqslant \widetilde{k_{i_n}^{n}} } C_n^k } = C_0 \cup \bigcup_{n} \bigcup_{ k\leqslant \widetilde{k_{i_n}^{n}} } C_n^k \cup \{ x \} .$$
where the upper dash denotes the closure operator.

Let us define $\psi \colon \mathbb{R}\to\mathbb{R}$ as $\psi(y) = \frac{y}{x}$, i.e. it is the scaling of the entire construction into the interval $[0,1]$.

We are ready to introduce the required map $\varphi \colon \mathcal{P}(\mathbb{N}^2) \to K([0,1]). $ First, let us define it on extreme arguments.
$$ \varphi(\emptyset) = \psi\left( C_0 \cup \bigcup_{n} \bigcup_{ k\leqslant \widetilde{k_{i_n}^{n}} } C_n^k \cup \{ x \} \right) ,$$
$$ \varphi(\mathbb{N}^2) = \psi\left( C_0 \cup \bigcup_{n} \bigcup_{ k\leqslant \widetilde{k_{i_n}^{n}} } E_n^k \cup \{ x \} \right) .$$
Now let us define $\varphi$ for any set $A\subseteq \mathbb{N}^2$. We will use auxiliary functions $\tau_n \colon \mathcal{P}(\mathbb{N}^2) \to \{0,1\}^\mathbb{N} $ such that $\tau_n (A) = (j_1,j_2,\dots) $ where $ j_i = 1 $ if $(n,i)\in A$ and $j_i=0$ otherwise. Let us go back to Remark \ref{rem}: there exists a set $\{ d_s\colon s\in \{0,1\}^{<\mathbb{N}} \}$ such that $D_n = \left\{ \sum_{i} d_{x|i} \colon x \in \{0,1\}^\mathbb{N},x_i=0\text{ for }i\leq i_n-1 \right\} $. Then define
$$ C_n^j(A) = \left\{ \sum_l d_{x|l} \cdot \delta^{ (j-1)( 1-\tau_n(A)(l) ) } \colon x \in \{0,1\}^\mathbb{N},x_i=0\text{ for }i\leq i_n-1 \right\}, $$
where $\tau_n(A)(l)$ denotes the $l$th term of $\tau_n(A)$, and finally shift $C_n^j(A)$ so that its midpoint coincides with the midpoint of $C_n^j$. Equivalently, using the notation from Remark \ref{def1}, $ C_n^j(A)=g^{\delta^{ (j-1)}}_{\tau_n(A)}(C_0\cap I^1_{i_n})$. Having all of those, put
$$ \varphi(A) = \psi \left( C_0 \cup \bigcup_{n} \bigcup_{ k\leqslant \widetilde{k_{i_n}^{n}} } C_n^k (A) \cup \{ x \} \right) .$$
Notice that this definition is consistent with the previously defined values at $\emptyset$ and $\mathbb{N}^2$.

Observe that the set $C_n^j(A)$ is an attractor of some IFS, for each $A$, $n$ and $j$. Indeed, without loss of generality we can assume that $C_n^j(A)\subseteq[0,1]$ (since a shift of an IFS attractor is still an IFS attractor). Given $x \in \{0,1\}^\mathbb{N}$ such that $x_i=0$ for $i\leq i_n-1$ let $\hat{x}^0,\hat{x}^1\in \{0,1\}^\mathbb{N}$ be given by:
 \begin{itemize}
\item $\hat{x}^0_i=\hat{x}^1_i=0$ for $i\leq i_n-1$,
\item $\hat{x}^0_{i_n}=0$,
\item $\hat{x}^1_{i_n}=1$,
\item $\hat{x}^0_{i}=\hat{x}^1_{i}=x_{i-1}$ for $i>i_n$.
\end{itemize} 
Let also $f_1,f_2:[0,1]\to[0,1]$ be functions such that:
\begin{itemize}
\item $f_1(\sum_l d_{x|l} \cdot \delta^{ (j-1)( 1-\tau_n(A)(l) ) })=\sum_l d_{\hat{x}^0|l} \cdot \delta^{ (j-1)( 1-\tau_n(A)(l) ) }$ for all $x \in \{0,1\}^\mathbb{N}$ such that $x_i=0$ for $i\leq i_n-1$,
\item $f_2(\sum_l d_{x|l} \cdot \delta^{ (j-1)( 1-\tau_n(A)(l) ) })=\sum_l d_{\hat{x}^1|l} \cdot \delta^{ (j-1)( 1-\tau_n(A)(l) ) }$ for all $x \in \{0,1\}^\mathbb{N}$ such that $x_i=0$ for $i\leq i_n-1$,
\item $f_1[[0,\min C_n^j(A)]]=f_1(\min C_n^j(A))$, 
\item $f_2[[0,\min C_n^j(A)]]=f_2(\min C_n^j(A))$,
\item $f_1[[\max C_n^j(A),1]]=f_1(\max C_n^j(A))$, 
\item $f_2[[\max C_n^j(A),1]]=f_1(\max C_n^j(A))$,
\item both $f_1$ and $f_2$ are linear on every open interval disjoint with $C_n^j(A)\cup[0,\min C_n^j(A)]\cup[\max C_n^j(A),1]$.
\end{itemize} 
From the above conditions it is obvious that $C_n^j(A)=f_1[C_n^j(A)]\cup f_2[C_n^j(A)]$. Observe also that both $f_1$ and $f_2$ are non-decreasing. We will show that both $f_1$ and $f_2$ are contractions. Note that actually from the last five items above it suffices to show that if $v,w\in C_n^j(A)$ then $|f_1(v)-f_1(w)|\leq\frac{1}{5}|v-w|$ and $|f_2(v)-f_2(w)|\leq\frac{1}{5}|v-w|$. Fix any $x,y\in \{0,1\}^\mathbb{N}$, $x\neq y$, such that $x_i=y_i=0$ for $i\leq i_n-1$. Without loss of generality we may assume that $y<_{lex} x$. Then also $\hat{y}^0|l\leq_{lex} \hat{x}^0|l$ and $\hat{y}^1|l\leq_{lex} \hat{x}^1|l$ for every $l$. Recall that for each $z\in \{0,1\}^\mathbb{N}$ we have $\sum_{n} d_{z|n} = \sum_{k, z_k = 1} \frac{9}{10^k} $ and  if $s,t\in\{0,1\}^{<\mathbb{N}}$, $s <_{lex} t, lh(s)=lh(t)$ then $d_s<d_t$ (see Remark \ref{rem}). Thus we have $f_1(x)>f_1(y)$ and:
$$|f_1(x)-f_1(y)|=f_1(x)-f_1(y)=\sum_l \left(d_{\hat{x}^0|l}-d_{\hat{y}^0|l}\right) \cdot \delta^{ (j-1)( 1-\tau_n(A)(l) ) }\leq$$
$$ \sum_l \left(d_{\hat{x}^0|l}-d_{\hat{y}^0|l}\right)=\sum_{k, \hat{x}^0_k = 1} \frac{9}{10^k}-\sum_{k, \hat{y}^0_k = 1} \frac{9}{10^k}=\frac{1}{10}\left(\sum_{k, x_k = 1} \frac{9}{10^k}-\sum_{k, y_k = 1} \frac{9}{10^k}\right)\leq$$
$$\frac{1}{5}\delta^{ j-1 }\left(\sum_{k, x_k = 1} \frac{9}{10^k}-\sum_{k, y_k = 1} \frac{9}{10^k}\right)=\frac{1}{5}\delta^{ j-1}\left(\sum_{l} d_{x|l}-\sum_{l} d_{y|l}\right)\leq$$
$$\frac{1}{5}\left(\sum_l d_{x|l} \cdot \delta^{ (j-1)( 1-\tau_n(A)(l) ) }-\sum_l d_{y|l} \cdot \delta^{ (j-1)( 1-\tau_n(A)(l) ) }\right).$$
This shows that $f_1$ is a contraction. For $f_2$ the reasoning is similar.

In the following for a fixed $A\subseteq\mathbb{N}^2$ we will refer to $\bigcup_{ k\leqslant \widetilde{k_{i_n}^{n}} } C_n^k (A)$ or to $\psi\left(\bigcup_{ k\leqslant \widetilde{k_{i_n}^{n}} } C_n^k (A)\right)$ as the $n$th section and to single $C_n^k (A)$ or $\psi(C_n^k (A))$ as member of the $n$th section.

Before presenting the remaining part of the proof, we would like to examine some conditions imposed on $\delta_n$ which follow from our construction.

\begin{itemize}
	\item[(1)] $$ \delta_{n}^{ \widetilde{k_{i_{n}}^{n}}-1} > \frac{80}{89}$$ (so we can apply Lemma \ref{delta}),
	\item[(2)] $$ |I_{i_{n}}^1| \cdot \delta_{n}^{ \widetilde{k_{i_{n}}^{n}} - 2 } \cdot 0.9 > | I_{i_{n}}^1 | - | I_{i_{n}}^1 | \cdot \delta_{n}^{ \widetilde{k_{i_{n}}^{n}} - 1 } .$$
	\item[(2')] $$ |I_{i_{n}}^1| \cdot \delta_{n}^{ \widetilde{k_{i_{n}}^{n}}-1 } \cdot 0.9 > \frac{1}{2} \cdot \left( | I_{i_{n}}^1 | - | I_{i_{n}}^1 | \cdot \delta_{n}^{ \widetilde{k_{i_{n}}^{n}} - 1 } \right). $$
\end{itemize}

The meaning behind condition (2) is the following: members of the $n$th section will not overlap even in the case of a maximal expansion. Indeed, the gap between the last two members of a section (which is always the smallest gap in a section), if they are maximally contracted, is of the same size as the left-hand side of the inequality, while the right-hand side of the inequality is an upper bound for the difference between lengths of that gap in cases of maximal contraction and maximal expansion:
$$ \frac{1}{2} \cdot \left( |I_{i_{n}}^1| - |I_{i_{n}}^1| \cdot \delta_{n}^{ \widetilde{k_{i_{n}}^{n}} - 2 } \right) + \frac{1}{2} \cdot \left( |I_{i_{n}}^1| - |I_{i_{n}}^1| \cdot \delta_{n}^{ \widetilde{k_{i_{n}}^{n}} - 1 } \right) < |I_{i_{n}}^1| - |I_{i_{n}}^1| \cdot \delta_{n}^{ \widetilde{k_{i_{n}}^{n}} - 1 } .$$
Similarly, condition (2') is introduced in order to assure that sections do not overlap. Let us notice that condition (1) implies (2) and (2'), i.e.:
$$ (1) \implies \delta_{n}^{ \widetilde{k_{i_{n}}^{n}} -1 } > \frac{10}{19} \implies (2). $$
$$ (1) \implies \delta_{n}^{ \widetilde{k_{i_{n}}^{n}} -1 } > \frac{10}{29} \implies (2'). $$

\begin{itemize}
	\item[(3)] $$ |I_{i_{n}}^1| \cdot \delta_{n}^{ \widetilde{k_{i_{n}}^{n}}-2 } \cdot 0.9 - \left( | I_{i_{n}}^1 | - | I_{i_{n}}^1 | \cdot \delta_{n}^{ \widetilde{k_{i_{n}}^{n}} - 1 } \right) > \frac{8}{10} \cdot |I_{i_{n}}^1|. $$ 
	\item[(3')] $$ |I_{i_{n}}^1| \cdot \delta_{n}^{ \widetilde{k_{i_{n}}^{n}}-1 } \cdot 0.9 - \frac{1}{2} \cdot \left( | I_{i_{n}}^1 | - | I_{i_{n}}^1 | \cdot \delta_{n}^{ \widetilde{k_{i_{n}}^{n}} - 1 } \right) > \frac{8}{10} \cdot |I_{i_{n}}^1|. $$
\end{itemize}
Similarly as before, the left-hand side of (3') is the length of a gap behind a section in the case of maximal expansion, while the right-hand side is the gap in the first Cantor section from the left (i.e. the longest one). From this condition the image under any weak contraction of the longest member of the $n$th section will not overlap with two sections on the right at once (and therefore neither will any other member of the $n$th section). In (3) the left-hand side is a lower bound for the gap between members of the $n$th section. Thus, (3) assures us that the image under any weak contraction of any member of the $n$th section will not overlap with two members of this section. Let us notice the following:
$$ (3') \iff \delta_{n}^{ \widetilde{k_{i_{n}}^{n}} - 1 } > \frac{13}{14}, $$
$$ \delta_{n}^{ \widetilde{k_{i_{n}}^{n}} - 1 } > \frac{18}{19} \implies (3) . $$

\begin{itemize}
	\item[(4)] \begin{align*}
	  |I_{i_{n}}^1|\cdot 0.9  < |I_{i_{n}}^1| \cdot \delta_{n}^{ \widetilde{k_{i_{n}}^{n}}-2 } \cdot 0.9 & - \left(|I_{i_{n}}^1| - |I_{i_{n}}^1| \cdot \delta_{n}^{ \widetilde{k_{i_{n}}^{n}} - 1 }\right)	+ \\ 
	  & + |I_{i_{n}}^1| \cdot 0.09 \cdot \delta_{n}^{ \widetilde{k_{i_{n}}^{n}} - 1 } . 
	  \end{align*}
\end{itemize}
Now the left-hand side is the maximal distance between two members of the $n$th section. The right-hand side is a bound for the minimal distance between two members of the $n$th section plus the $|I_{i_{n}}^1| \cdot \delta_{n}^{ \widetilde{k_{i_{n}}^{n}}-1 } \cdot 0.09$ factor. Following condition (1) and Lemma \ref{delta}, the image of any member $C_n^j(A)$ of the $n$th section under a weak contraction can intersect at most two consecutive sets of the form $E^{i_n}_i(C_n^{j'}(A))$ for $i\leq 4$ (see Remark \ref{def1} for the definition of those sets). The minimal distance between $\min C_n^{j'}(A)$ and $E^{i_n}_2(C_n^{j'}(A))$ (as well as between $\max C_n^{j'}(A)$ and $E^{i_n}_3(C_n^{j'}(A))$) is at least $|I_{i_{n}}^1| \cdot \delta_{n}^{ \widetilde{k_{i_{n}}^{n}}-1 } \cdot 0.09$. Therefore, condition (4) (together with (3)) guarantees that for any weak contraction $f$ such that $f[\varphi(A)]\subseteq\varphi(A)$ and $f[C_n^j(A)]\subseteq C_n^{j'}(A)$, either $f[C_n^j(A)]$ is a subset of the left-hand half of $C_n^{j'}(A)$ or $f[C_n^{j-1}(A)]\cup f[C_n^{j+1}(A)]\subseteq C_n^{j'}(A)\cup C_n^{j'+1}(A)$.
Similarly, either $f[C_n^j(A)]$ is a subset of the right-hand half of $C_n^{j'}(A)$ or $f[C_n^{j-1}(A)]\cup f[C_n^{j+1}(A)]\subseteq C_n^{j'}(A)\cup C_n^{j'-1}(A)$. Let us notice the following.
$$ \delta_{n}^{ \widetilde{k_{i_{n}}^{n}} - 1 } > \frac{190}{199} \implies (4). $$

\begin{itemize}
	\item[(4')] $$ |I_{i_{n}}^1|\cdot 0.9 < |I_{i_{n}}^1| \cdot \delta_{n}^{ \widetilde{k_{i_{n}}^{n}}-1 } \cdot 0.9 - \frac{1}{2} \cdot \left( | I_{i_{n}}^1 | - | I_{i_{n}}^1 | \cdot \delta_{n}^{ \widetilde{k_{i_{n}}^{n}} - 1 } \right) + |I_{i_{n}}^1| \cdot \delta_{n}^{ \widetilde{k_{i_{n}}^{n}}-1 } \cdot 0.09. $$
\end{itemize}

Condition (4') is analogous to (4), but instead of considering the minimal distance between two members of the $n$th section we take the minimal gap behind the $n$th section. After computations it turns out that 
$$ (4') \iff \delta_{n}^{ \widetilde{k_{i_{n}}^{1}}-1 } > \frac{140}{149}. $$

Now we continue our proof. We will show that $\varphi$ is continuous. Fix any $A\subseteq\mathbb{N}^2$ and $r\in(0,1)$. Note that $\{B\subseteq\mathbb{N}^2:A\cap F=B\cap F\}$ is open in $\mathcal{P}(\mathbb{N}^2)$, for every finite set $F\subseteq\mathbb{N}^2$. Thus, we wish to find a finite set $F\subseteq\mathbb{N}^2$ such that $B\cap F=A\cap F$ implies $\varphi(B)\in B(\varphi(A),r)$, where $B(\varphi(A),r)$ denotes the ball in $K([0,1])$ of radius $r$ and centered at $\varphi(A)$. Since $\lim_{n\to\infty} \sum_{i>n} a_i = 0 $, there is $m\in\mathbb{N}$ such that $\varphi(B)\in B(\varphi(A),r)$ whenever $A\cap(\{1,\ldots,m\}\times\mathbb{N})=B\cap(\{1,\ldots,m\}\times\mathbb{N})$. Also, for any $x \in \{0,1\}^\mathbb{N}$ the series $\sum_{n} d_{x|n} $ is convergent. Therefore for any $n<m$ only the initial interval of $\tau_n(A)$ matters. This shows that $\varphi$ is continuous. 

If $A\subseteq \mathbb{N}^2$ is such that 
$$ \forall_n^\infty \: \exists_m^\infty\: (n,m) \notin A, \qquad (*) $$
then we would like to obtain that $\varphi(A) \in wIFS^1$. The sentence $(*)$ means that $\exists_{n_0}\:\forall_{n>n_0}\:\exists_m^\infty\:(n,m)\notin A$. Let us denote the union of first $n_0$ sections and $\psi(C_0)$ by $B$, i.e.:
$$B = \psi \left( C_0 \cup \bigcup_{n \leqslant n_0 } \bigcup_{ k\leqslant \widetilde{k_{i_n}^{n}} } C_n^k (A) \right).$$
The set $B$ is a union of $\psi(C_0)$ (which is an attractor for the contractions $f_1(x)=\psi(\frac{x}{10})$ and $f_2(x)=\psi(\frac{x}{10}+\frac{9}{10})$) and finitely many pairwise disjoint sets of the form $\psi(C_n^k (A))$. Since each $C_n^k (A)$ is an attractor for some IFS, so is each $\psi(C_n^k (A))$. We can additionally require that each of those contractions sends $[0,\min \psi(C_n^k (A)))$ on the same point as $\min \psi(C_n^k (A))$ and $(\max \psi(C_n^k (A)),1]$ on the same point as $\max \psi(C_n^k (A))$. Therefore there exists $k$ and an iterated function system $\{ f_1, f_2, \dots, f_k \}$ (consisting of regular contractions), such that $B = \bigcup_{i\leqslant k} f_i[B]$. Without the loss of generality $\bigcup_{i\leqslant k} f_i[\varphi(A)]\subseteq B$ (it suffices to require that $f_i[\varphi(A)\setminus B]=f_i(\max B)$ for all $i\leq k$).

Let $f$ be such that the following conditions hold:
 \begin{itemize}
\item[1.] $f(1) = \psi(x) = 1, $
\item[2.] $\forall_{n\geqslant n_0} \: f \left[ C_{n}^{\widetilde{k_{i_n}^n}} \right] = C_{n+1}^1 $,
\item[3.] $\forall_{n > n_0} \: \forall_{j < \widetilde{k_{i_n}^n} } f \left[ C_{n}^j \right] = C_{n}^{j+1}$,
\item[4.] $ f\left[ 0, \min \psi \left( C_{n_0}^{k_{i_{n_0}}^{n_0}} \right) \right] = \min \psi \left( C_{n_0}^{k_{i_{n_0}}^{n_0}} \right) $,
\item[5.] $f$ is a weak contraction.
\end{itemize} 
Then we would have $ \varphi (A) = f\left[ \varphi(A) \right] \cup \bigcup_{i\leqslant k} f_i\left[ \varphi(A) \right]$. Why is it possible to pick such $f$? In fact only item 3. is problematic as it could contradict item 5. Now we will show that this is not a problem.

Let $n>n_0$ and $j< \widetilde{k_{i_n}^n}$. Also, fix $x,y \in \{0,1\}^\mathbb{N}, x >_{lex} y$. There exists $k$ such that $ x|k \neq y|k $. Additionally, there exists $m>k$ such that $(n,m) \notin A$ (since $n>n_0$). Therefore $\tau_n(A)(m) = 0.$ Then
$$ \sum_{l} d_{x|l} \cdot \delta^{ (j-1)\cdot (1 - \tau_n(A)(l) ) } - \sum_{l} d_{y|l} \cdot \delta^{ (j-1)\cdot (1 - \tau_n(A)(l) ) } = $$
$$ =  \sum_{l} \left( d_{x|l}-d_{y|l} \right) \cdot \delta^{ (j-1)\cdot (1 - \tau_n(A)(l) ) } = $$
$$ =  \sum_{l\neq m} \left( d_{x|l}-d_{y|l} \right) \cdot \delta^{ (j-1)\cdot (1 - \tau_n(A)(l) ) } + \left( d_{x|m} - d_{y|m} \right) \cdot \delta^{(j-1)} \geqslant  $$
$$ \geqslant  \sum_{l \neq m} \left( d_{x|l}-d_{y|l} \right) \cdot \delta^{ j\cdot (1 - \tau_n(A)(l) ) } + \left( d_{x|m} - d_{y|m} \right) \cdot \delta^{(j-1)} >  $$
$$ >  \sum_{l \neq m} \left( d_{x|l}-d_{y|l} \right) \cdot \delta^{ j\cdot (1 - \tau_n(A)(l) ) } + \left( d_{x|m} - d_{y|m} \right) \cdot \delta^{j}. $$
(the last inequality is true as $d_{x|m} - d_{y|m} > 0$ by $x >_{lex} y $ -- see Remark \ref{rem}). This part is concluded.

The following is left to be shown:
$$ \forall_{A\subseteq\mathbb{N}^2}\:\exists_n^\infty\:\forall_m^\infty\:\: (n,m)\in A \implies\varphi(A)\notin wIFS^1. $$

Let us fix $A\subseteq\mathbb{N}^2$ such that $\exists_n^\infty\:\forall_m^\infty\:\: (n,m)\in A$ and assume otherwise: that $\varphi(A)$ is an attractor for a system of $k$ weak contractions $f_1, f_2, \dots, f_k$. Let us define the following sequences of sets:
$$ K_n = \psi \left( \bigcup_{l\leqslant \widetilde{ k_{i_n}^n } } C_n^l(A) \right),$$
$$ L_n = \varphi(A) \cap [ 0, \min K_n ) ,$$
$$ R_n = \varphi(A) \cap ( \max K_n, 1 ] ,$$
(so $K_n$ is the $n$th section).

\textit{Observation 1.} For any $j>k$ the set $\bigcup_{ i\leqslant k } f_i[ L_j ] $ intersects at most a fifth of all $ \psi \left( C_j^l(A) \right) $. It follows from the fact that 
$$ \widetilde{k_{i_j}^j} = k_{i_j}^j \cdot 5 \cdot j > k_{i_j}^j \cdot 5 \cdot k $$
and $L_j$ consists of $k_{i_j}^j$ Cantor sets longer than $\psi\left( C_j^l(A) \right)$ (the same amount for every $l\leqslant \widetilde{k_{i_j}^j} $) each of which cannot overlap two members of $j$th section (by (3')).

\textit{Observation 2.} If $ f_i\left[ R_j \right] \cap K_j \neq \emptyset $ for some $i\leq k$ and $j\in\mathbb{N}$, then the whole set $ f_i\left[ R_j \right]$ is a subset of one member of $j$th section. This follows from the fact that each $R_j$ has gaps between its members which are smaller than the smallest gap between members of $j$th section.

\textit{Observation 3.} If for some $i\leq k$ there is $j\in\mathbb{N}$ such that $ f_i\left[ K_j \right] \subseteq K_j \cup L_j$, then $ f_i\left[ K_j \cup R_j \right] \subseteq K_j \cup L_j$. 

By Observation 3., it is impossible that for every $i\leqslant k$ there exists $j_i$ such that $ f_i\left[ K_{j_i} \right] \subseteq K_{j_i} \cup L_{j_i}$. Indeed, otherwise for $n > k,j_1, \dots, j_k$ the set $K_n$ would not be entirely covered (by Observation 1.), since in this case
$$K_n\cap\bigcup_{i\leq k}f_i[\varphi(A)]=K_n\cap\bigcup_{i\leq k}f_i[L_n].$$

Let us pick $n>k$ such that $\forall_m^\infty \: (n,m) \in A $ and for every $i \leqslant k$ if there exists such $j_i$ that $f_i \left[ K_{j_i} \right] \subseteq K_{j_i} \cup L_{j_i} $, then $n>j_i$. If no such $j_i$ exists, then instead pick any $n>k$ such that simply $\forall_m^\infty \: (n,m) \in A $. We also require for $n$ to be sufficiently large that 
$$ \widetilde{k_{i_{n}}^{n}} > \lceil \frac{1}{2} \widetilde{k_{i_{n}}^{n}}   \rceil + \lceil \frac{1}{5} \widetilde{k_{i_{n}}^{n}}   \rceil + k.$$
Since $\widetilde{k} \to \infty$, finding such $n$ is possible.

We will show that $K_n$ cannot be covered by $\bigcup_{i\leq k }f_i[\varphi(A)]$. Let us notice that the set $\bigcup_{i \leqslant k }f_i[L_n] $ intersects at most a fifth of the members of $K_n$ (Observation 1.). Since for each $i$ the set $f_i[R_n]$ intersects at most one member of $K_n$  (Observation 2.), the union $\bigcup_{i \leqslant k }f_i[R_n] $ covers at most $k$ of them. Thus, there are still 
$$ \widetilde{k_{i_{n}}^{n}} - \lceil \frac{1}{5} \widetilde{k_{i_{n}}^{n}}   \rceil - k>\lceil \frac{1}{2} \widetilde{k_{i_{n}}^{n}}   \rceil$$
members of $K_n$ to be covered by $\bigcup_{i\leq k}f_i[K_n]$.

Fix $i\leqslant k$. By the choice of $n$, either $f_i[K_n]\subseteq L_n$ (if there is $j_i$ as above) or $f_i[K_n]\cap R_n\neq\emptyset$ (see Observation 3.). In the latter case, let $\psi(C_n^j)$ be the member of the $n$th section which goes onto $R_n$ (i.e., $f_i[\psi[C_n^j]]\subseteq R_n$ -- note that $f_i[\psi[C_n^j]]$ cannot intersect both $R_n$ and $K_n\cup L_n$ by (3')). Suppose that $j-2\geq 1$. Then $\psi(C_n^{j-1})$ either goes onto $R_n$ as well or onto a subset of $\psi(C_n^{\widetilde{k^n_{i_n}}})$ (recall that $f_i[\psi(C_n^{j-1})]$ cannot intersect two members of the $n$th section by condition (3)). As we want to cover $K_n$, we can assume that the latter case holds. By conditions (1) and (4) imposed on $\delta$ and Lemma \ref{delta}, $\psi(C_n^{j-2})$ also goes onto a subset of $\psi(C_n^{\widetilde{k^n_{i_n}}})$. The same reasoning works for $\psi(C_n^{j+1})$ and $\psi(C_n^{j+2})$ instead of $\psi(C_n^{j-1})$ and $\psi(C_n^{j-2})$ (whenever $j+2\leq \widetilde{k^n_{i_n}}$). If we will carry on this reasoning, we will see that $f_i[K_n]$ can cover at most a half of all members of $K_n$, however it will be necessary the \textit{right-hand half}. 

The previous paragraph shows that $\bigcup_{i\leq k}f_i[K_n]$ can cover at most a half of all members of $K_n$. Thus, by the choice of $n$, $K_n$ cannot be covered by $\bigcup_{i\leq k }f_i[\varphi(A)]$. This finishes the entire proof.
\end{proof}

We are ready to prove our main result.

\begin{theorem}
	The set wIFS$^d$ is $G_{\delta\sigma}$-hard in $K([0,1]^d)$, for every $d\in\mathbb{N}$.
\end{theorem}

\begin{proof}
There is a continuous function $\varphi\colon \mathcal{P}(\mathbb{N}^2) \to K([0,1])$ such that 
	$$ \varphi^{-1}[wIFS^1] = \mathcal{P}(\mathbb{N}^2) \setminus \{ A \colon \exists_n^\infty\: \forall_m^\infty\:\: (n,m) \in A \} $$
(by Lemma \ref{main_theorem}). 
	
	Define $\hat{\varphi}\colon \mathcal{P}(\mathbb{N}^2) \to K([0,1]^d)$ by $\hat{\varphi}(A)=(\varphi(A),0,\ldots,0)$ for all $A\in\mathcal{P}(\mathbb{N}^2)$. It is easy to check that $\hat{\varphi}$ is continuous. Moreover, 
	$$ \hat{\varphi}^{-1}[wIFS^d] = \mathcal{P}(\mathbb{N}^2) \setminus \{ A \colon \exists_n^\infty\: \forall_m^\infty\:\: (n,m) \in A \} .$$
Indeed, $B\in\{ A \colon \exists_n^\infty\: \forall_m^\infty\:\: (n,m) \in A \}$ implies $\hat{\varphi}(B)\notin wIFS^d$, since if $f:[0,1]^d\to[0,1]^d$ is a weak contraction then so is $f\upharpoonright[0,1]\times\{0\}^{d-1}$. On the other hand, if $B\notin\{ A \colon \exists_n^\infty\: \forall_m^\infty\:\: (n,m) \in A \}$ then $\hat{\varphi}(B)$ is in wIFS$^d$, since for each weak contraction $f:[0,1]\to[0,1]$ the map $g:[0,1]^d\to[0,1]^d$ given by $g(x_1,\ldots,x_d)=(f(x_1),0,\ldots,0)$ is a weak contraction as well. Thus, by Lemma \ref{fsigmadelta-complete} the set wIFS$^d$ is $G_{\delta\sigma}$-complete.
\end{proof}

Unfortunately, we were not able to solve the following problem.

\begin{question}
Let $d\in\mathbb{N}$. Is the set wIFS$^d$ Borel?
\end{question}

\newcommand{\nosort}[1]{}

\begin{center}
\flushleft{{\sl Addresses:}} \\
Pawe\l{} Klinga \\
Institute of Mathematics \\
University of Gda\'nsk \\
Wita Stwosza 57 \\
80 -- 952 Gda\'nsk \\
Poland\\
e-mail: pawel.klinga@ug.edu.pl
\end{center}

\begin{center}
\flushleft 
Adam Kwela \\
Institute of Mathematics \\
University of Gda\'nsk \\
Wita Stwosza 57 \\
80 -- 952 Gda\'nsk \\
Poland\\
e-mail: akwela@mat.ug.edu.pl
\end{center}

\end{document}